\documentclass[12pt]{article}

\usepackage{amsthm}
\usepackage{amsmath}
\usepackage{natbib}
\usepackage{apalike}
\newtheorem{thm}{Theorem}
\begin{document}
\title{ Unbiasedness and Bayes Estimation}

\author{Siamak Noorbaloochi\thanks{Research was supported by IIR 12-340 –HSR\&D  grant, 
 Department of Veterans Affairs, Veterans Health Administration, 
Office of Research and Development, Health Services Research and Development}\\ 
 HSR\& D Center for Chronic Disease Outcomes Research\\
  VA Health Care System\\
  Department of Medicine\\
        University of Minnesota\\
        Minneapolis,  Minnesota \\
and\\
Glen Meeden \\
School of Statistics\\
University of Minnesota\\
School of Statistics\\
313 Ford Hall\\
224 Church ST S.E.\\
University of Minnesota\\
 Minneapolis, MN 55455-0460
}
\date{November, 2015}

\maketitle

\newpage

\begin{abstract}
 Assuming squared error loss, we show that  finding  unbiased estimators
and Bayes estimators  can be treated as using a pair of linear operators
that operate between two Hilbert spaces. We note that
   these integral operators are adjoint  and then investigate some
  consequences of this fact.
\end{abstract}

\vspace{.2in}


\newpage

\section{Introduction}

Statistical inference is used to produce 'plausible' data-based
assertions and rules about a partially unknown population. Two
well-adopted and seemingly adverse inference plausibility criteria are
unbiasedness and being Bayes.  The task of the current note is to
explore further the relationship between these two procedures by
treating them as operators that represent our inference procedures.

Let $p_{\theta}(\cdot)$ be the density of $X$ for the parameter
$\theta$. Let $\pi$ be the prior density for $\theta$, with the sample
space being denoted by $\mathcal{X}$ and the parameter space by
$\Theta$.  Suppose $\gamma$, some real-valued function defined on
$\Theta$, is to be estimated using $X$. A data-based rule $\delta$ is
said to be mean unbiased for $\gamma$ if $E_{\theta}\delta(X) =
\gamma(\theta)$ for all $\theta \in \Theta $. Under squared error
loss, for a given prior $\pi$ over $\Theta$, a rule $\delta_{\pi}$ is
a Bayes rule (against squared error loss) for $\gamma$ if
$E(\gamma(\theta)|x)=\delta_{\pi}(x)$ for all $x\in\mathcal{X}$.

 Lehmann (1951) \nocite{leh51} proposed a
generalization of the above notion of unbiasedness which takes into
account the loss function for the problem. Noorbaloochi and Meeden
(1983) \nocite{n-m83} proposed a generalization of Lehmann's definition which depends,
in addition, on a prior distribution $\pi$ for $\theta$.

In this note, restricting to squared error loss, corresponding to a
given prior $\pi$, a Hilbert space of all square integrable
real-valued functions of $X$ and $\theta$ is constructed. Given the induced inner
products and norms we observe that  unbiasedness and being Bayes are adjoint
operators. From this fact we derive some orthogonality relationships
between Bayes and unbiased estimators and the functions they are
estimating.

\section{Notation}
We begin with some  notation and the spaces that we will use to develop the discussion. Let
\[
{\cal H }_{\pi} =\{ h(x,\theta): \int \int h^2(x,\theta) p_{\theta} (x)
\pi (\theta) dx\, d\theta < \infty \}
\]
be the space of all square-integrable
real-valued functions of $(x,\theta)$.
Note ${\cal H}_{\pi}$ becomes a  Hilbert space
when it is equipped with the inner product
\[
(h_1,h_2)=\int\int h_1(x,\theta)h_2(x,\theta) p_\theta (x)\pi(\theta)dx\, d\theta
\]
Let $\| h\|_{\pi} = \sqrt{(h,h)}$ denote the norm of $h$. We include the subscript
$\pi$ to remind us that ${\cal H_{\pi}}$ does depend on $\pi$.
We assume that $p_{\theta}(x)$ is in the above space.

There are two linear subspaces of $\cal H_{\pi}$ which are of particular interest. The first
is
\[
\Gamma_{\pi}=\{\gamma(\theta):\int\gamma^2(\theta)\pi(\theta)d\theta\, <\infty\}
\] and the second is
\[
\Delta_{m}=\{\delta(x):\int\delta^2(x) m(x)dx\,<\infty \}
\]
where $m(x)=\int{p_\theta(x)\pi(\theta)\,d\theta}$.

The set $\Gamma_{\pi}$ is a Hilbert subspace of ${\cal H }_{\pi}$ with
the induced weighted inner product
$(\gamma_1,\gamma_2)_{\pi}=\int_{\Theta}\gamma_1(\theta)\gamma_2(\theta)\pi(\theta)\,d\theta$.
Similarly,
$\Delta_{m}$ is a Hilbert subspace with the induced weighted inner
product $(\delta_1,\delta_2)_m=\int_{\cal
  X}\delta_1(x)\delta_2(x)m(x)dx$. We also notice that, provided the
interchange of order of integration is permitted, for any
$\delta\in\Delta_m$ we have: $\text{Var}_{\theta}\delta(X)<\infty$ and
hence $E_{\theta}\delta(X)<\infty$ for all $\theta$ in the support of
$\pi$. If the support of $\pi$ is all of  $\Theta$, all members of
$\Delta_m$ are unbiased estimators of their expectations and, for any
$\delta\in\Delta_m$:
\[\int_{\Theta}E_{\theta}\delta(X)\pi(\theta)\,d\theta=E_m\delta(X)<\infty\]
\noindent implying that when $\pi(\theta)>0$ for all $\theta$, the set
of all  estimable functions, which we will denote by $\Gamma_e$, is
a subset of $\Gamma_{\pi}$.The assumed square-integrability of the
likelihoods and the Holder inequality imply that for any
$\gamma\in\Gamma_{\pi}$:
\[\int_{\Theta}\gamma(\theta)p_{\theta}(x)\pi(\theta)\,d\theta<\infty\]
\noindent and hence all have Bayes estimators. With the above
notation, the Euclidian distance between any $\gamma\in\Gamma_{\pi}$
and $\delta\in\Delta_{m}$ is $\| \delta -\gamma\|_{\pi}^2 =
r(\delta,\gamma;\pi)$, which is the Bayes risk associated with the
pair.

Let us define the  operator, $\mathcal{U}$
$$\mathcal{U}:\Delta_{m}\rightarrow\Gamma_{\pi}$$
\noindent as the \emph{unbiasedness operator} if
 for a given $\gamma$, $\mathcal{U}\delta=\gamma$ if and only if
\[
r(\delta,\gamma;\pi) = \inf_{\gamma^{\prime}\in\Gamma_{\pi}} r(\delta,\gamma^{\prime};\pi)
\]
For  mean-unbiasedness,  $\mathcal{U}$ can be defined through the integral transform:
$$\mathcal{U}:\Delta_{m}\rightarrow\Gamma_{\pi}  \,\,\,\text{if and only if}
\,\,\,E_{\theta}(\delta(X))=\gamma(\theta)\,\,\text{for all}\,\,\theta$$
\noindent Hence in this case, the operator $\mathcal{U}$ is a linear operator.

Similarly, the Bayes operator: $\mathcal{B}_{\pi}$ may be defined as:
$$\mathcal{B}_{\pi} :\Gamma_{\pi}\rightarrow\Delta_{m}$$
\noindent where for each $\gamma \in \Gamma_{\pi}$, $\mathcal{B}_{\pi} \gamma=\delta_{\pi}$
if and only if
\[
r(\delta_{\pi},\gamma;\pi) = \inf_{\delta^{\prime}\in\Delta_{m}} r(\delta^{\prime},\gamma;\pi)
\]
\noindent For squared error loss, the Bayes operator corresponds to
the linear operator:
$$\mathcal{B}_{\pi} :\Gamma_{\pi}\rightarrow\Delta_{m} \,\,\,\text{if and only if}
\,\,\,E(\gamma(\theta)|x)=\delta(x)\,\,\text{for all}\,\,x\in
\mathcal{X} $$

We are now ready to state the main observation of the manuscript. Throughout we always
assume that we are dealing with a fixed prior $\pi$.

\section{The relationship between unbiasedness and being Bayes}

Given the above setup we now show that
the Bayes and unbiasedness operators,
  $\mathcal{B}_{\pi}$ and $\mathcal{U}$ are the adjoint operators of each
  other. That is, for any $\gamma\in\Gamma_{\pi}$ and any $\delta\in
  \Delta_{m}$ we have
\begin{equation}
(\gamma,\mathcal{U}\delta)_{\pi}=(\mathcal{B}_{\pi} \gamma,\delta)_m
\label{eq:adj}
\end{equation}
The proof is easy, i.e,
 \begin{eqnarray*}
(\gamma,\mathcal{U}\delta)_{\pi}&=& \int_{\Theta}\gamma(\theta)E_{\theta}\delta(X)\pi(\theta)\,d\theta
=\int_{\cal X}\delta(x)[\int_{\Theta}\gamma(\theta)f_{\theta}(x)\pi(\theta)\,d\,\theta]dx \\
&=&\int_{\cal X}\delta(x)[\int_{\Theta}\gamma(\theta)\frac{f_{\theta}(x)\pi(\theta)} {m(x)}\,d\theta]m(x)dx
=(\delta,E(\gamma(\theta)\mid x))_m\\
&=&(\delta,\mathcal{B}_{\pi} \gamma)_m
\end{eqnarray*}

We note that the unbiased operator is independent of the chosen
prior and $\mathcal{U}$ simultaneously is the adjoint of all
$\mathcal{B}_{\pi}$ for \emph{all} priors with $\Theta$ support.

We denote the range of $\mathcal{U}$ by $\mathcal {R}(\mathcal{U})$. This
 is the set of all functions  in $\Gamma_{\pi}$ which have
an unbiased estimator. We denote the
range of $\mathcal{B}_{\pi} $ by
$\mathcal {R}(\mathcal{B}_{\pi} )$. This  is the
set of all Bayes estimators (with respect to  squared error loss)
in $\Delta_{m}$. In addition we have the null spaces of these two operators.
$\mathcal{N}(\mathcal{U})$ is the set of all unbiased estimators of zero.
If the model is complete this will contain just one function.
$\mathcal{N}(\mathcal{B}_{\pi})$  is the set of all functions with zero as  their Bayes
estimator. It is a basic result of functional analysis (see Rudin (1991)) 
\nocite{rud91} that from equation
\ref{eq:adj} we can write
\begin{align}
\Gamma_{\pi} & = {\mathcal{R}(\mathcal{U})}\bigoplus \mathcal{N}(\mathcal{B}_{\pi} ) \\
\Delta_{m} & ={{\mathcal R}(\mathcal{B}_{\pi} )}\bigoplus {\mathcal N}(\mathcal{U})
\end{align}
The first equation implies that  every member of $\Gamma_{\pi}$ can be orthogonally
decomposed into a function with an unbiased estimator plus a
function whose  Bayes estimator is zero and the second equation implies that
every member of $\Delta_{m}$ can be orthogonally
decomposed into a Bayes estimator (of some $\gamma$) plus an unbiased
estimator of zero and both these decompositions are unique. As far as we know this has
never been noted before and shows that the notions of being unbiased and being Bayes
are  more closely entwined than previously thought.

The first equation shows that given any $\gamma \in \Gamma_{\pi}$ there exist a unique
$\gamma_e \in \mathcal{R}(\mathcal{U})$ and a
unique $\alpha \in \mathcal{N}(\mathcal{B}_{\pi})$
such that
\[
\gamma(\theta) = \gamma_e(\theta) + \alpha(\theta) \quad \text{for}\; \theta \in \Theta
\]
So every function  will have an unbiased estimator if and only if the  only function
whose Bayes estimator is the zero function is zero function. Furthermore when
$\alpha$ is not the trivial function we see that the Bayes estimator of $\gamma$
must also be the Bayes estimator of $\gamma_e$.

The second equation shows that given  $\delta \in \Delta_{m}$ there exists a unique
$\delta_{\pi} \in {\mathcal R}(\mathcal{B}_{\pi} )$ and
a unique $\delta_0 \in {\mathcal N}(\mathcal{U})$
such that
\[
\delta(x) = \delta_{\pi}(x) + \delta_0(x) \quad \text{for}\; x \in \mathcal{X}
\]
So every decision function will be a Bayes estimator for some $\gamma$ if and only if
the only unbiased estimator of the zero function is the zero function. Also, the
orthogonal decompositions imply that

\begin{align}
\|\delta\|_m^2 & = \|\delta_{\pi}\|_m^2 + \|\delta_0\|_m^2 \\
\|\gamma\|_{\pi}^2 &  = \|\gamma_e\|_{\pi}^2 + \|\alpha\|_{\pi}^2
\end{align}

\section{Some  Consequences}

\begin{thm} Suppose $\mathcal{R}(\mathcal{U})$ is a proper subset of $\Gamma_{\pi}$.
Let $\delta \in \Delta_{m}$ be the Bayes estimator for some function
$\gamma \in \Gamma_{\pi}$ which does not belong to $\mathcal{R}(\mathcal{U})$.
Then there exists a unique $\gamma_e \in
{\mathcal{R}(\mathcal{U})}$ and an unique $\alpha_0 \in\mathcal{N}(\mathcal{B}_{\pi})$
such that $\delta$ is the Bayes estimator of $\gamma_e$. In addition, the Bayes
risk of $\delta$ when estimating $\gamma_e$ is strictly less than its Bayes
risk when estimating $\gamma$.

\end{thm}

\begin{proof}
  For a given $\gamma$, by equation 2 it  can be written as
\[
\gamma= \gamma_e + \alpha
\]
where $\gamma_e \in \mathcal{R}(\mathcal{U})$ and
$\alpha \in \mathcal{N}(\mathcal{B}_{\pi})$.  Then for each $x \in \mathcal{X}$
we have
\begin{align*}
E(\gamma(\theta)|x) &=E(\gamma_e(\theta)|x) + E(\alpha(\theta |x) \\
                    &= E(\gamma_e(\theta)|x)
\end{align*}
since $E(\alpha(\theta)|x) =0$. We note that the decomposition in
equation 2 is prior-dependent. Indeed, $\gamma_e$ is the projection of
$\gamma$ into $\Gamma_e$, that
is: $$\|\gamma(\theta)-\gamma_e(\theta)\|_{\pi}=min_{\gamma_e'\in\Gamma_e}\|\gamma(\theta)-\gamma_e'(\theta)\|_{\pi}$$

To prove the second part we note that

$$\|\delta_{\pi}-\gamma\|_{\pi}^2 = \|\delta_{\pi}-\gamma_e\|_{\pi}^2+ \|\alpha_0\|_{\pi}^2
-2(\delta_{\pi}-\gamma_e,\alpha_0)_{\pi}$$
\noindent but $E_{\pi}(\gamma_e\alpha_0)=0$ by orthogonality of
$\gamma_e$ and $\alpha_0$, and
$$E_{\pi}E_{\theta}(\delta_{\pi}(X)\alpha_0(\theta))=E_m(\delta_{\pi}(X)E(\alpha_0(\theta)|X))=0$$
\noindent since $\alpha_0(\theta)\in\mathcal{N}(\mathcal{B}_{\pi})$. Therefore,
$$\|\delta_{\pi}-\gamma\|_{\pi}^2 = \|\delta_{\pi}-\gamma_e\|_{\pi}^2+ \|\alpha_0\|_{\pi}^2$$
\end{proof}

\begin{thm} Let $\gamma$ be a member of $\mathcal{R}(\mathcal{U})$ and let $\delta$
be its Bayes estimator. Let $\lambda(\theta)= E_{\theta}\delta(X)$ and  $b(\theta)=
\gamma(\theta) - \lambda(\theta)$, the bias of $\delta$ as an estimator of $\gamma$.
Then the Bayes risk of $\delta$ as an estimator of $\gamma$ is $(b,\gamma)_{\pi}$.
\end{thm}

\begin{proof}
 The Bayes risk of the Bayes rule $\delta_{\pi}$ for estimating $\gamma$ is:
\begin{align*}
r(\delta_{\pi},\gamma)& = E_{\pi}E_{\theta}(\delta(X)-\gamma(\theta))^2 \\
        & =
E_m\delta_{\pi}^2(X)+E_{\pi}\gamma^2(\theta)-E_{\pi}E_{\theta}(\delta_{\pi}(X)\gamma(\theta))
-E_{\pi}E_{\theta}(\delta_{\pi}(X)\gamma(\theta))
\end{align*}
but
$$E_{\pi}E_{\theta}(\delta_{\pi}(X)\gamma(\theta))=E_m(\delta_{\pi}(X)E(\gamma(\theta)|X))
=E_m\delta_{\pi}^2(X)$$
 and  similarly conditioning on $\theta$
$$E_{\pi}E_{\theta}(\delta_{\pi}(X)\gamma(\theta))=
E_{\pi}(\gamma(\theta)E(\delta_{\pi}(X)|\theta))=
E_{\pi}(\gamma(\theta)\lambda(\theta))$$
 Substituting these into the previous equation and simplifying we have
\begin{align*}
r(\delta_{\pi},\gamma)& = E_{\pi}\gamma^2(\theta)-E_{\pi}(\gamma(\theta)\lambda(\theta)) \\
   & =E_{\pi}\gamma^2(\theta)+E_{\pi}(\gamma(\theta)( b(\theta)-\gamma(\theta))) \\
& = E_{\pi}\gamma(\theta)b(\theta) \\
 & = (b,\gamma)_{\pi}
\end{align*}
\end{proof}

Note that an immediate corollary is the well known fact
that if an estimator is both unbiased and Bayes
for some $\gamma$ then its Bayes risk is zero.

Given a model, a prior and a function to be estimated, the Bayes risk of the Bayes
rule  is a measure of the informativeness of our
inferences. The smaller the size of the Bayes risk
 the more informative
is our "best" estimator about  the function being estimated. 
If this minimum is
large then the Bayes estimator is not very informative. (For further discussion
on this point see see Raiffa and
Schlaiffer(1962), DeGroot(1962,1984) and  Ginebra(2007)). The previous theorem
quantifies the relationship between the Bayes risk and bias and shows that
there will only be a ``good'' Bayes estimator when its bias is small.
\nocite{Raif1961} \nocite{Degroot62} \nocite{Degroot84} \nocite{Gin2007}

Therefore, in order to reduce the minimum Bayes risk of the Bayes estimator of the function
being estimated, one has to reduce the bias of the Bayes rule. This can
be achieved by increasing the sample size, as the following argument
shows. For the rest of this note we assume that given $\theta$,
$X_1,\ldots,X_n$ are independent and identically distributed random
variables. Furthermore $\|h\|_{\pi,n}^2$ will denote the norm for the
$n$ fold problem.

\begin{thm}
 For a sample of size one
let $U \in \Delta_{\pi,1}$ be an unbiased estimator of $\gamma \in \Gamma_{\pi}$
and $U_n = \sum_{i=1}^n U(X_i)/n$. Let $\delta_{\pi,n}$ be the
Bayes estimator of $\gamma$ based on the sample of size $n$.
If  $\|U_n - \gamma\|_{\pi,n}^2 \rightarrow 0$ as $ n\rightarrow \infty $  then
\begin{enumerate}
\item[i.] $\|\delta_{\pi,n} -\gamma\|_{\pi,n}^2 \rightarrow 0 $ as $n \rightarrow \infty$
\item[ii.] $\|U_n - \delta_{\pi,n}\|_{\pi,n}^2 \rightarrow 0 $ as $n \rightarrow \infty$
\end{enumerate}
\end{thm}

\begin{proof}
Let $\tau$ be the Bayes risk of $U$ for estimating $\gamma$ when $n=1$. Then
the Bayes risk of $U_n$ is just $\tau/n$. But the Bayes risk for $\delta_{\pi,n}$
for estimating $\gamma$ is no greater than the Bayes risk of $U_n$ so part $i$
follows.

It is easy to see that
\[
\|U_n - \gamma\|_{\pi,n}^2 = \|U-\delta_{\pi,n}\|_{\pi,n}^2 +  \|\delta_{\pi,n} -\gamma\|_{\pi,n}^2
\]
by adding and subtracting $\delta_{\pi,n}$ inside the lefthand side and then multiplying
out and observing that the cross product term is zero by conditioning on the data.
Now part $ii$ follows from part $i$ and the above equation
 because both of the terms involving $\gamma$ go to zero as
 $n \rightarrow \infty$.

\end{proof}

The second part of the theorem implies for large $n$ that a Bayesian whose prior
is $\pi$ believes with high probability that their estimator will be close to
the unbiased estimator. Note another Bayesian with a different prior believes the
same thing. So they both expect agreement as the sample size increases.
This result is somewhat in the spirit of  one in
Blackwell and Dubins (1962).\nocite{b-d62}

It is well know that Bayes estimators are usually consistent and tend to agree as the
sample size increases.
Most of the standard arguments for the consistency of Bayes estimators start with
considering the asymptotic distribution of the posterior distribution,
see for example  Johnson (1970). \nocite{joh70} These arguments are closely
related to the asymptotic behavior of the maximum likelihood estimator and
can be technically quite complex.
For a  discussion of this point and more references
see  O'Hagan (1994).\nocite{oha94} Here however we tied
the asymptotic behavior of the Bayes estimator to that of an unbiased estimator.
The lack of assumptions and the
relative simplicity of this argument comes at the cost of being unable to
say anything directly about the asymptotic behavior of the posterior distribution.
But it does underline the relationship between being Bayes and unbiasedness.

\section{Two simple examples}

In this section we consider two simple examples to demonstrate in more detail
the relationships between the two concepts.

{\textbf{Example 1. }} Let $X$ be a random variable, with
probability mass function in the family:
$p_{\theta}(x)=\frac{1}{2},\,\, \text{for}\,\,
x=1,\,\,\,=\frac{1-\theta}{4}, \,\,\text{for}\,\, x=2,\,
\text{and}\, = \frac{1+\theta}{4}, \text{for}\,\, x=3$, with
$0<\theta\le 1$. For the uniform prior, the posterior is:
$\pi(\theta|x)=1, \,\text{when}\, x=1, =2(1-\theta),\,
\text{for}\,\, x=2,\,\, \text{and} \,\,=\frac{2(1+\theta)}{3},\,\,\text{for}\,\,
x=3$, with $m(x)=\frac{1}{2},\frac{1}{8}\, \text{and}\,\frac{3}{8}$,
respectively, for $x=1,2,\,\text{and}\,3$. For this example,
$\Gamma_{\pi}=\{\gamma(\theta):\int_0^1\gamma^2(\theta)\,d\theta<\infty\}$
and $\Delta_{m}={\mathcal{R}}^3$. It is easy to check that
$\mathcal{R}(U)=\{l(\theta):l(\theta)=a+b\theta\,\, a,b\in
\mathcal{R}\,\}$ and
$\mathcal{N}(U)=\mathcal{Z}=\{(a,-a,-a): a\in \mathcal{R} \}$. Letting
$c=\int_0^1\gamma(\theta)\,d\theta$ and
$d=\int_0^1\theta\gamma(\theta)\,d\theta$, we see that $
{\mathcal{R}(\mathcal{B}_{\pi}
  )}=\{(c,2(c-d),\frac{2(c+d)}{3}):\,\,c,d\in \mathcal{R}\}$.
We also have that
$$\mathcal{N}(\mathcal{B}_{\pi} )=\{\gamma(\theta): \int_0^1\gamma(\theta)
\,d\theta=0\,\,\text{and}\, \int_0^1\theta\gamma(\theta)\,d\theta=0 \}$$
Note that
$(\delta,z)_{m}=\sum_1^3\delta_{\pi}(x)z(x)m(x)=\frac{1}{2}ac-\frac{2(ac-d)}{8}-\frac{3}{8}
\frac{2(ac+d)}{3}=0$
confirming the orthogonality of Bayes rules and unbiased estimators of
zero. Also, for any
$\gamma(\theta)\in \mathcal{N}(\mathcal{B}_{\pi} )$ and any $l(\theta)=a+b\theta\in
\mathcal{R}(\mathcal{U})$:
$$(\gamma,l)_{\pi}=\int_0^1\gamma(\theta)(a+b\theta)
\,d\theta=a\int_0^1\gamma(\theta)+b\int_0^1\theta\gamma(\theta)=0$$
and  thus confirming  the orthogonality of functions which have an unbiased estimator
and those whose Bayes estimate is zero.

{\textbf{Example 2. }} Let $X$ be a Bernoulli,
with success probability $\theta$ and the  prior is uniform over
$(0,1)$. The class of linear functions in $\theta$ is the subspace of
estimable functions, $\Gamma_e$. The function $e^{\theta}$ does not have an 
unbiased estimator
but its Bayes estimator is $E(e^{\theta}|x)=2(e-2)$, when $x=0$
and is $=2$, when $x=1$. However, projecting $e^{\theta}$
into this subspace using the least square procedure by solving:
$$\begin{cases}
    a_0+a_1E_{\pi}\theta=E_{\pi}e^\theta\\
    a_0E_{\pi}\theta+a_1E_{\pi}\theta^2=E_{\pi}\theta e^\theta ,
  \end{cases}$$ 
  
yields the the function
  $\gamma_e(\theta)=(4e-10)+(18-6e)\theta$. This function has the same
  Bayes estimator as $e^{\theta}$ and the Bayes estimator of the
  approximation error,
  $(4e-9)+(17-6e)\theta+\sum_2^{\infty}\theta^k/k!$, is zero. However
  the same Bayes estimator of $e^{\theta}$ and $(4e-10)+(18-6e)\theta$
  have different Bayes risks: $1/2(3e^2-16e+21)\approx 0.1626$ and
  $2e^2-12e+18\approx 0.1587$, respectively.

Suppose now we wish to estimate the function $1-\theta^2$. Its Bayes estimator,
say $\delta^*$, estimates 5/6 when $X=0$ and 1/2 when $X=1$.
It is easy to check
that the only polynomials of degree 2 whose Bayes estimator is the zero function are of
the form $a/6 - a\theta + a\theta^2$ for some real number $a$. Hence the decomposition
\[
1-\theta^2 = (7/6 - \theta) + (-1/6 + \theta - \theta^2)
\]
breaks it up into the sum of a function with an unbiased estimator  and a function
whose Bayes estimator is the zero function. So $\delta^*$ is also the Bayes estimator
of the function $7/6 - \theta$.

More generally, in the binomial case of size n, if we let,
$P_0(\theta)=1$ and for $k=1,\dots,n$ let $P_k(\theta)$ be orthonormal polynomials
with respect to the inner product of $\Gamma_{\pi}$, which together forms a
basis for the space of the functions which have unbiased estimators,
then the Bayes estimator of any $\gamma\in\Gamma_{\pi}$ is

\[E(\gamma(\theta)|x)=\sum_{k=0}^{n}(\gamma,P_k)_{\pi}E(P_k(\theta)|x)=E_{\pi}\gamma(\theta)+
\sum_{k=1}^{n}(\gamma,P_k)_{\pi}E(P_k(\theta)|x)\]

\noindent This shows that all the Bayes estimates belong to the linear
space spanned by the Bayes estimates of this basis and moreover,
$\gamma(\theta)-\sum_{0}^{n}(\gamma,P_k)_{\pi}P_k(\theta)$ has zero as
its Bayes estimator. Note that the space of unbiasedly estimable
functions is independent of the prior, but the projection of a given
function $\gamma$ (without an unbiased estimator) into this subspace
depends on the prior through the induced inner product.

\section{Final remarks}

We see from equation (2)  that a function $\gamma$ has an unbiased estimator
if  and only if it is orthogonal to every function of the parameter which has
the zero function as its Bayes estimator. From equation (3) we see
that a function $\delta$ is a Bayes estimator of some $\gamma$ if and only if
it is orthogonal to all unbiased estimators of the zero function.

This becomes clearer if we think about the situation where both
 $\Theta$ and $\mathcal{X}$ are  finite. Let $K$ be the number of elements in $\Theta$
and $N$ be the number of elements in $\mathcal{X}$. Then $\gamma$ has an  unbiased
 estimator $\delta$ if
\begin{equation}
\sum_{x \in \mathcal{X}} (\delta(x) - \gamma(\theta))p_{\theta}(x) = 0 \quad\text{for}\;
    \theta \in \Theta
\label{eq:unb}
\end{equation}
while $\delta$ is Bayes for $\gamma$ if
\begin{equation}
\sum_{\theta \in \Theta}(\delta(x)-\gamma(\theta)) \pi(\theta|x)=0
\quad\text{for}\; x \in \mathcal{X}
\label{eq:bay}
\end{equation}
These are two  systems of linear equations. The first has
$K$ equations in $N$ unknowns, the $\delta(x)$'s, while the second
has $N$ equations in $K$ unknowns, the $\gamma(\theta)$'s. 
We may arrange the members of the statistical model into the
$K\times N$ stochastic matrix having as its rows
the $K$ probability mass functions, $P=(p_{\theta_i}(x_j))$.
The collection of posteriors yield a 
$N$ by $K$ stochastic matrix $\Pi=(\pi(\theta_i|x_j))$. Here $\mathcal{U}$
is the matrix $P$ and $\Pi$ is $\mathcal{B}_{\pi}$. We notice that the two
matrices, $P$ and $\Pi$, need not be square and hence necessarily
are not projection matrices. It is the
relationship between ranks of $P$ and $\Pi$ which determines the form of the
four linear subspaces described above and results in the close relationship between
being Bayes and being unbiased.

In this note we have seen that 
considering  inference procedures as operators acting between the
data and the parameter space can lead to new insights about the 
relationship between being Bayes and unbiasedness. This suggests that
there could be  other insights that arise from this perspective.

\bibliography{refnew}
\bibliographystyle{apalike}

\end{document}